\documentclass[11pt]{amsart}
\usepackage{amscd,amsmath,amssymb,amsfonts,ifthen}

\usepackage{bbold}
\usepackage{graphicx}
\usepackage{url,mathrsfs,euscript}
\usepackage[all]{xypic}

\def\sO{\mathcal O}
\def\sD{\mathcal D}

\def\sC{\mathcal C}

\newcommand{\aro}{\longrightarrow}
\newcommand{\bb}[1]{\mathbf{#1}}
\newcommand{\ot}{\otimes}

\newcommand\cg{\mathcal G}

\newcommand{\loc}{\rm loc}
\def\et{\text{\rm \'et}}

\numberwithin{equation}{section}
\newtheorem{thm}[equation]{Theorem}
\newtheorem{prop}[equation]{Proposition}
\newtheorem{lem}[equation]{Lemma}

\theoremstyle{definition}
\newtheorem{defn}[equation]{Definition}
\newtheorem{rmk}[equation]{Remark}

\usepackage{charter}
\usepackage{eulervm}

\date{30 November 2016}
\author[P. H. Hai]{Ph\`ung H\^o Hai } 
\address[PHH]{Institute of Mathematics, Vietnam Academy of Science and Technology, 18 Hoang Quoc Viet, Cau Giay, Hanoi, Vietnam}

\email{phung@math.ac.vn}
\thanks{The research of PHH is funded by Vietnam National Foundation for  Science and Technology Development (NAFOSTED). Part of this work has been carried out during his visit at the Max-Planck Institute for Mathematics, Bonn, he would like to thank the Institute for its hospitality and financial support.}

\author[J. P. dos Santos]{Jo\~ao Pedro P. dos Santos}

\address[JPdS]{Institut de Math\'ematiques de Jussieu -- Paris Rive Gauche, 4 place Jussieu, 
Case 247, 75252 Paris Cedex 5, France}

\email{joao-pedro.dos-santos@math.jussieu.fr}

\title[Action of $\pi^{\et}$ on $\pi^o$]{The action of the \'etale fundamental group scheme on the connected component of the essentially finite one}

\begin{document}
\begin{abstract}
We follow the pattern in \cite[Section 4]{Otabe} to define an action of the \'etale fundamental group \emph{scheme} $\pi^\et(X)$ on the local component of the essentially finite fundamental group scheme $\pi^{\mathrm{EF}}(X)$ of Nori. We show that the associated  representation is faithful when $X$ is a curve of genus $\geq 2$. 
\end{abstract} 
\maketitle  
 
\parskip2pt

\section{Introduction} 
Let $X$ be a connected, proper and  reduced algebraic scheme over a perfect field $k$, and $x$ a $k$-rational point of $X$. In his seminal work \cite{Nori76}, M. V. Nori detected that a full subcategory of the category of vector bundles on $X$ can be used to produce, via the Tannakian correspondence, an affine group scheme $\pi^{\rm EF}(X,x)$ over $k$ which, colloquially speaking, classifies torsors with finite structural group. 
If the characteristic of $k$ is positive, $\pi^{\rm EF}(X,x)$ possesses two relevant canonical  quotients: $\pi^{\et}(X,x)$, which is the largest pro-\'etale  one, and $\pi^{\rm loc}(X,x)$, which is  the largest local one. By considering the kernel of the morphism $\pi^{\rm EF}(X,x)\to\pi^{\et}(X,x)$, we then obtain another local affine group scheme, call it $\pi^{\rm EF}(X,x)^{o}$, and the question concerning the relation between $\pi^{\rm EF}(X,x)^o$ and $\pi^{\rm loc}(X,x)$ naturally arises.  

In \cite{EHS08}, the authors explained that $\pi^{\rm loc}(X,x)$ in fact only accounts for a   small portion of $\pi^{\rm EF}(X,x)^o$ by showing that the latter actually contains information about $\pi^{\rm loc}(X')$
 for \emph{all} ``geometric'' \'etale coverings $X'\to X$ 
 (see Theorem 3.5 of op. cit. for a precise statement). 
 Further, in \cite{EH10} it was noticed that $\pi^{\rm EF}(X,x)$ is a semi-direct product of $\pi^{\rm EF}(X,x)^o$ with $\pi^{\et}(X,x)$, and, when $X$ is a smooth projective curve, the action of  $\pi^{\et}(X,x)$ on $\pi^{\rm EF}(X,x)^o$ is trivial if and only if $X$ has genus at most 1 (see Corollary 2.3 and Propostion 2.4 of op. cit.).
 
The work \cite{EHS08} inspired Otabe \cite{Otabe} to show that, in case $k$ is of characteristic zero,  his  ``semi-finite fundamental group scheme'' $\pi^{\rm EN}(X,x)$ \cite[Section 2.4]{Otabe} produces a \emph{faithful} action of $\pi^{\et}(X,x)$ on its unipotent radical  provided $X$ is a smooth curve of genus at least two (see \cite[Theorem 4.12]{Otabe}). 

We wish to demonstrate here that Otabe's point of view can give interesting information in the case of positive characteristic. Our main finding is that the action of $\pi^{\et}(X,x)$ on $\pi(X,x)^o$ is faithful if $X$ is a geometrically connected, smooth and projective curve of genus at least two. 
See Section \ref{28.11.2016--3}, specially Theorem \ref{28.11.2016--2}. 

We now briefly describe the contents of this article. In Section \ref{26.07.2016--1} we review Nori's theory and some of its later developments. In Section \ref{12.08.2016--1} we slightly modify the presentation leading to Theorem 3.5 of \cite{EHS08} so that we can easily state and prove our main result, Theorem \ref{26.09.2016--1} of Section \ref{28.11.2016--3}. It is perhaps useful to note that Theorem \ref{26.09.2016--1} has a more portable consequence, which we present as Theorem \ref{28.11.2016--2}. 
The proof of Theorem \ref{26.09.2016--1} requires an exercise which is carried out on Section 
\ref{12.09.2016--1}.

\subsection*{Notations, conventions and generalities} 

\subsection{Conventions}\label{notations} 
 
\paragraph{{\it On vector bundles}} A vector bundle is a locally free coherent sheaf of finite rank. If $x:\mathrm{Spec}\,K\to X$ is a point of a scheme $X$ and $V$ is a vector bundle on $X$, we write $V|_x$ for the $K$-vector space $x^*V$. A vector bundle $V$ over $X$ is said to be trivial if it is isomorphic to some $\sO_X^{\oplus r}$. 

\paragraph{\emph{On group schemes}} For an affine group scheme $G$ over a field $k$, we write $k[G]$ instead of $\Gamma(G,\mathcal O_G)$.  Given an affine group scheme $G$, the category of all its \emph{finite dimensional} representations is denoted by $\mathrm{Rep}_k(G)$.  
An arrow $q:G\to H$  of affine group schemes is called a quotient morphism if it is faithfully flat. We use constantly the fact that $q:G\to H$ is a quotient morphism if and only if the associated arrow $k[H]\to k[G]$ is injective \cite[Chapter 14]{waterhouse}.  

\paragraph{\emph{On Abelian varieties}} For an abelian variety $A$, we let $[m]:A\to A$ stand for multiplication by $m$. The kernel of $[m]$ is denoted  by  $A[m]$. 

\subsection{Generalities on adjunctions in the category of affine group schemes}\label{16.09.2016--1}

Let $\bb G$ be the category of affine group schemes. In this section, we explain how to treat in more robust fashion the process of ``taking the largest quotient having a certain property''. 

We first note that
\begin{itemize}\item[($\star$)] $\bb G$ is stable under all small limits (use the standard criterion \cite[V.2, Corollary 2]{maclane}), 
\item[($\star\star$)]  and that each arrow $f:G\to H$ can be decomposed uniquely as 
\[
G\stackrel{q}{\aro} I \stackrel{i}{\aro} H,
\]
where $i$ is a closed embedding and $q$ is a quotient morphism. 
\end{itemize}

Let $u:\bb A\to \bb G$ be a full subcategory of $\bb G$ enjoying the ensuing properties:
\begin{enumerate}\item[P1.] The category $\bb A$ is small complete and $u$ preserves all small limits.
\item[P2.] If $A$ belongs to $\bb A$ and $i:H\to A$ is a closed embedding, then $H$ also belongs to $\bb A$. 
\item[P3.] If $A$ belongs to $\bb A$ and $q:A\to H$ is a quotient morphism, then $H$ also belongs to $\bb A$.
\end{enumerate}

Under such conditions, it is a direct consequence of Freyd's theorem \cite[V.6, Theorem 2]{maclane} that $u$ has a left adjoint $G\mapsto G^{\mathbf A}$. 
Furthermore: 
\begin{lem}\label{13.09.2016--1}The unit morphism $\eta_G:G\to u(G^{\bb A})$ is always a quotient morphisms while the co-unit $\varepsilon_A:(uA)^{\bb A}\to A$ is always an isomorphism. 
\end{lem}
\begin{proof}Let $G\stackrel q\to I\stackrel i\to  u(G^{\bb A})$  be the decomposition of $\eta_G$ predicted by ($\star\star$). Then,  $I\in\bb A$ and it follows that $q:G\to I$ is universal from $G$ to $u$, so that $i$ is an isomorphism. 
The second claim follows immediately from \cite[V.3, Theorem 1]{maclane}.  
\end{proof}

This justifies the following standard terminology: 
\begin{defn}\label{23.09.2016--3}If $G$ is an affine group scheme, then the arrow $G\to G^{\bb  A}$ is called the largest quotient of $G$ in $\bb A$. 
\end{defn}

Let now $v:\bb B\to \bb G$ be a second 
subcategory enjoying P1--P3. From Lemma \ref{13.09.2016--1} and stability under quotients, we conclude that $B^{\bb A}\in \bb B$ for all $B\in\bb B$; one easily sees that $(-)^{\bb A}:\bb B\to\bb A\cap\bb B$ is left adjoint to the inclusion $\bb A\cap\bb B\to \bb B$.  
This being so, the composition 
\[
\bb G\stackrel{(-)^{\bb B}}\aro \bb B\stackrel{(-)^{\bb A}}\aro \bb A\cap\bb B
\]
is adjoint to the inclusion $\bb A\cap\bb B\to\bb G$ since ``left adjoint of a composition is the composition of the left adjoints'' \cite[V.8, Theorem 1]{maclane}. 
Consequently, employing \cite[V.1, Corollary 1]{maclane} we have 
\begin{lem}Let $\bb A$ and $\bb B$ be categories as above. Then, the compositions   
\[\xymatrix{
\bb G\ar[r]^{(-)^{\bb A}}& \bb A \ar[r]^-{(-)^{\bb B}}& \bb A\cap \bb B
}\quad\text{and}\quad\xymatrix{
\bb G\ar[r]^{(-)^{\bb B}}& \bb B \ar[r]^-{(-)^{\bb A}}& \bb A\cap \bb B
}\]
are naturally isomorphic. Moreover, they are also naturally isomorphic to 
\[
\xymatrix{\bb G\ar[rr]^{(-)^{\bb A\cap\bb B}}&& \bb A\cap\bb B} 
\]
\qed 
\end{lem}

As is customary, if $\bb  A$ is the category of abelian affine group schemes, respectively local affine group schemes,  then $G^{\bb A}$ is denoted by $G^{\rm ab}$, respectively $G^{\rm loc}$. They are then, in the spirit of Definition \ref{23.09.2016--3} above, called the largest abelian quotient, respectively the largest local quotient, of $G$.

\section{The essentially finite fundamental group scheme}\label{26.07.2016--1}
In this section, we make a leisurely introduction to the essentially finite group scheme; it serves mainly to help us establish notation and to introduce the reader to our mode of thought. Besides the seminal text \cite{Nori76}, the reader should consult \cite{EHS08} for detailled information. 

In what follows, $k$ stands for a perfect field of characteristic $p>0$. 
Let $X$ be a proper, reduced and connected algebraic scheme over $k$. 
In \cite{Nori76}, Nori introduced two important classes of vector bundles: the (now called) Nori-semistables and the finite. A vector bundle $V$ on $X$ is said ot be Nori-semistable if it becomes semistable and of degree zero when pulled back along any non-constant morphism $\gamma:C\to X$ from a smooth and projective curve (see the Definition after Proposition 3.4 in \cite{Nori76}). 
The second class, the finite vector bundles, are those $V$ for which the set 
\[
 \left\{\begin{array}{c}\text{isomorphism classes of indecomposable} \\ \text{direct summands of $V^{\otimes 1},V^{\otimes2},\ldots$}
\end{array}\right\}
\]
is \emph{finite} (see the Definition after Lemma 3.1 in \cite{Nori76}).
It turns out that all finite vector bundles are Nori-semistable and that the category of Nori-semistables -- any morphism of vector bundle being an arrow -- is abelian. This fact allows one to consider all the Nori-semistables of the form $W/W'$, where $W'\subset W$ are both subobjects of a common finite $V$, and show that the resulting category, 
with the evident tensor product, is a tensor category over $k$ in the sense of \cite[1.2]{deligne90}. This is the category of \emph{essentially finite} vector bundles, which is denoted in what follows by $\sC^{\rm EF}(X)$. 

Given a $k$-point $x$ of $X$, the functor $V\mapsto V|_x$ (see section \ref{notations}) from $\sC^{\rm EF}(X)$ to $k\textbf{-vect}$
is exact and faithful, so that the main result of Tannakian theory \cite[2.11, p.130]{dm} constructs an affine group scheme over $k$, usually called the Nori or \emph{essentially finite fundamental group scheme} $\pi(X,x)$, and an equivalence of tensor categories 
\[
\sC^{\rm EF}(X)\stackrel{\sim}{\longrightarrow} \mathrm{Rep}_k(\pi(X,x)),\quad V\longmapsto V|_x.
\]

Let us now elaborate on an useful notion. Given $V\in\sC^{\rm EF}(X)$, let   $\langle V\rangle_\otimes$ stand  for the full subcategory of $\sC^{\rm EF}(X)$ whose objects are subquotients of finite direct sums of vector bundles of the form $V^{\otimes a}\otimes V^{*\otimes b}$. Then, 
\[
\bullet|_x:\langle V\rangle_\otimes\longrightarrow \mathrm{Rep}_k(\pi(X,x))
\]
defines an equivalence between $\langle V\rangle_\otimes$ and the category $\mathrm{Rep}_k(\pi(X,V,x))$ of a certain quotient  $\pi(X,V,x)$ of $\pi(X,x)$ \cite[2.21,p.139]{dm}. This quotient turns out to be a \emph{finite} group scheme, a fact which can be grasped by looking at the definition of a finite vector bundle and \cite[2.20(a), p.138]{dm}.

The full subcategory of $\sC^{\rm EF}(X)$ consisting of those $V$ for which $\pi(X,V,x)$ is \'etale, respectively local, will be denoted by $\sC^{\et}(X)$, respectively $\sC^{\rm loc}(X)$. Accordingly, objects of $\sC^{\et}(X)$, respectively of $\sC^{\loc}(X)$, are called \'etale, respectively local, vector bundles. 
By means of the criterion \cite[Proposition 2.21, p.139]{dm} and the fact that \'etale and local finite group schemes are stable under quotient morphism,  the functor $\bullet|_x$ induces an equivalence between $\sC^{\et}(X)$, respectively $\sC^{\rm loc}(X)$, and a \emph{quotient} $\pi^{\et}(X,x)$, respectively $\pi^{\rm loc}(X,x)$, of $\pi(X,x)$. Needless to say, the affine group scheme $\pi^{\loc}(X,x)$, respectively $\pi^{\et}(X,x)$, is a projective limit of finite and local group schemes,  respectively finite and \'etale group schemes. 

The relation between  $\pi^{\et}(X,x)$ and its celebrated predecessor, the \'etale fundamental group of \cite{SGA1} is quite simple: Let $\overline k$ be an algebraic closure of $k$, and write $\overline X=X\otimes_k\overline k$. 
Then, using the obvious geometric point $\overline x:\mathrm{Spec}\,\overline k\to\overline X$, we construct the geometric fundamental group $\pi_1(\overline X,\overline x)$ of $\overline X$. Since $\overline x$ actually comes from a $k$-rational point, $\pi_1(\overline X,\overline x)$ has a continuous action of $\mathrm{Gal}(\overline k/k)$, and by the construction of \cite[II, \S5, no. 1.7]{DG} we can associate to $\pi_1(\overline X,\overline x)$ a profinite group scheme. This is  $\pi^{\et}(X,x)$. As we shall have no use for this characterization here, we omit the verifications.
(Note that this relation is incorrectly stated in \cite[2.34]{dm} and partially explained in \cite[Remarks 2.10]{EHS08}.) 

We end this section with a result which is left implicit in most works on the subject.

\begin{lem}\label{27.11.2016--1}Let $E$ be a vector bundle over $X$, and $K$ be a finite and separable extension of $k$.  Then $E$ is essentially finite if and only if $E\ot K$ is essentially finite over $X\ot K$. Moreover, the same statement is true if we replace ``essentially finite'' by ``local'' or ``\'etale''. 
\end{lem}
\begin{proof}Only the ``if'' statement needs attention, so assume that $E\ot K$ is essentially finite. 
We can therefore find a finite group scheme $G$ (over $K$), a $G$-torsor $P\to X\ot K$, and a monomorphism $E\ot K\to \mathcal O_{P}^{\oplus r}$. Now, according to 
\cite[Chapter II, Propsoition 5, p.89]{Nori82}, $P$ can be chosen to come from $X$, that is,  $P=P_0\ot K$, where $P_0\to X$ is a torsor under a certain finite group scheme. Consequently, we obtain a monomorphism of $\mathcal O_X$-modules $E\to \mathcal O_{P_0}^{\oplus r}\ot K$; as $E$ is certainly Nori-semistable, we conclude that  $E$ is essentially finite. 
The proof of the last claim follows the same method, since we can replace $G$ with a local, or \'etale finite group scheme.  
\end{proof}

\section{The kernel of $\pi(X)\to\pi^{\et}(X)$}\label{12.08.2016--1} We maintain the notations and terminology of section \ref{26.07.2016--1}, but omit reference to the base point $x$  in speaking about fundamental group schemes. 
In what follows we briefly review some results of \cite{EHS08}, including one of its main outputs, Theorem 3.5 on p. 389. In fact, we shall,  with an eye to future applications, use a different path to arrive at \cite[Theorem 3.5]{EHS08}; see the discussion after  Definition \ref{12.09.2016--2} below.

We begin with generalities and remind the reader that we ignore in notation the dependence of the chosen base point $x$. Given  $V\in\sC^{\rm EF}(X)$, \emph{an inverse} of the equivalence 
\[
\bullet|_x:\langle V\rangle_\otimes\aro \mathrm{Rep}_k(\pi(X,V))
\] 
constructed on Section \ref{26.07.2016--1} produces a principal $\pi(X,V)$-bundle 
\[
\psi_V:  X_V\longrightarrow X
\]
together with a $k$-point $x_V$ on the fibre of $\psi_V$ above $x$. Moreover, our inverse equivalence is just the contracted product functor 
\begin{equation}\label{contracted_product}
\mathscr L_{X_V}: \mathrm{Rep}_k(\pi(X,V)
)\aro
\langle V\rangle_\otimes.
\end{equation}
(See \cite[I.4.4.2]{saavedra} for the existence an inverse to $\bullet|_x$ which is a tensor functor and  \cite[11ff]{Nori76} for the construction of $X_V$.)

Let us fix $V\in\sC^{\et}(X)$ and simplify notations by writing 
\[ \boxed{
X':=X_V,\quad \psi=\psi_V,\quad G=\pi(X,V),\quad x'=x_V. }
\]
Two simple features of $X'$ are immediately remarked: $X'$ is reduced and proper ($\psi$ is finite and  \'etale), and  $X'$ is ``Nori''-reduced, that is,  
$\Gamma(X',\mathcal O_{X'})=k$, see \cite[Proposition 3, p. 87]{Nori82}. We are then allowed to consider $\sC^{\rm loc }(X')$, and set out to investigate its  relation to $\sC^{\rm EF}(X)$.  
Using the proof of Theorem 2.9 in \cite{EHS08} (see also the paragraph preceding Lemma 2.8 on p. 384), we can say the following. 

\begin{thm}For each $E'\in \sC^{\rm EF}(X')$, the vector bundle $\psi_*(E')$ is also essentially finite on $X$. \qed
\end{thm}
Hence, we obtain a functor  
\[
\psi_*:\sC^{\rm loc}(X')\longrightarrow \sC^{\rm EF}(X) 
\]
which, it turns out, allow us to understand the category of representations of the kernel 
\[\mathrm{Ker}\,\pi(X,x)\longrightarrow \pi^\et(X,x).\]
Until now, this is exactly the point of view in \cite{EHS08}; let us start making minor changes.  

\begin{defn}\label{12.09.2016--2}Given any finite set  $S$ of objects in $\sC^{\rm EF}(X)$, we let $\langle S\rangle_\otimes$ stand for the full subcategory \[\langle \oplus_{W\in S}W\rangle_\otimes\]
of $\sC^{\rm EF}(X)$. If $S$ is an arbitrary set of objects in $\sC^{\rm EF}(X)$, we let  $\langle S\rangle_\otimes$ stand for the full subcategory having  
\[
\bigcup_{\text{$s\subset S$ finite}} \left\langle s  \right\rangle_\otimes
\]
as  objects.  
\end{defn}

We now  apply the above definition to the set of objects of
$\psi_*\sC^{\rm loc}(X')$. Let  
\[\pi(X,\sC^{\rm loc}(X'))\]
stand for the \emph{quotient} of $\pi(X)$ obtained by means of the category 
\[
\langle\psi_{*}\sC^{\rm loc}(X')\rangle_\otimes
\]
and the fibre functor $\bullet|_x$ through the basic result \cite[2.21, p.139]{dm}.   

\begin{prop}\label{10.09.2016--1}The following claims are true. 
\begin{enumerate}
\item 
 A vector bundle $E\in\sC^{\rm EF}(X)$ belongs to $\langle\psi_*\sC^{\rm loc}(X')\rangle_\ot$ if  and only if  $\psi^*E$ belongs to $\sC^{\loc}(X')$.

\item The vector bundle $V$ belongs to $\langle\psi_{*}\sC^{\rm loc}(X')\rangle_\otimes$ and the resulting morphism 
\[
\pi(X\,;\,\sC^{\rm loc}(X'))\aro \pi(X,V)=G
\]
is a quotient morphism. 

\item Each $E\in\sC^{\rm loc}(X)$ belongs to $\langle\psi_{*}\sC^{\rm loc}(X')\rangle_\otimes$ and the resulting morphism 
\[
\pi(X\,;\,\sC^{\rm loc}(X'))\aro \pi^{\rm loc}(X) 
\]
is a quotient morphism. In particular, $\pi^{\loc}(X)$ is the largest local quotient of $\pi(X\,;\,\sC^{\rm loc}(X'))$. 
 
\end{enumerate}
\end{prop}

\begin{proof} 
(1) The proof goes as that of  \cite[Lemma 2.8, p.384]{EHS08}. 
 Let  $E=\psi_*(E')$, where $E'$ is a local vector bundle. Using the cartesian square 
\[
\xymatrix{X'\times G \ar[d]_{\mathrm{pr}}\ar[r]^-{\alpha} & X'\ar[d]^\psi \\ X'\ar[r]_\psi& X,   }
\] 
where $\alpha$ is the action morphisms, we conclude that $\psi^*E\simeq \mathrm{pr}_*\alpha^*E'$. But, after a possible extension of the base field, $X'\times G$  becomes a disjoint sum of copies of $X'$ while $\mathrm{pr}_*\alpha^*E'$ becomes a sum of vector bundles of the shape $g^*E'$, where $g\in\mathrm{Aut}(X')$. Hence, $\psi^*E\in \sC^{\loc}(X')$. (Here we have implicitly used Lemma \ref{27.11.2016--1}.)
For a general $E\in\langle\psi_*\sC^{\rm loc}(X')\rangle_\otimes$, the definition says that $E\in\langle \psi_*(E')\rangle_\otimes$ for some $E'\in\sC^{\rm loc}(X')$. But then, as $\psi^*:\sC^{\rm EF}(X)\to\sC^{\rm EF}(X')$ is an exact tensor functor, $\psi^*E$ belongs to $\langle\psi^*(\psi_*E')\rangle_\otimes$, which is a subcategory of $\sC^{\rm loc}(X')$, as $\pi(X)\to\pi^{\rm loc}(X)$ is a quotient morphism.

Now let $E\in\sC^{\rm EF}(X)$ be such that $\psi^*E$ belongs to $\sC^{\loc}(X')$. Since $\psi$ is faithfully flat,   the ``unit'' 
$E\to\psi_{*}\psi^*(E)$ is a monomorphism, and consequently $E$ belongs to $\langle\psi_*(\psi^*E)\rangle_\ot$. By definition, this says that $E$ lies in $\langle\psi_*\sC^{\loc}(X')\rangle_\ot$. 

(2) The first claim is a consequence of  (1) and the fact that $\psi^*V$ is trivial. Since $\langle\psi_*\sC^{\loc}(X')\rangle_\ot$ is a full subcategory of $\sC^{\rm EF}(X')$ which is stable under subquotients, the standard criterion \cite[2.21, p.139]{dm} guarantees the veracity of the second statement once applied to the inclusion $\langle V\rangle_\ot\subset\langle\psi_*\sC^{\loc}(X')\rangle_\ot$.

(3) This is again a simple application of (1) and the   criterion \cite[2.21, p.139]{dm}.
\end{proof}

At this point, we wish to describe the kernel of  
\[
\pi(X\,;\,\sC^{\rm loc}(X'))\aro G=\pi(X,V ),
\]
which is the statement paralleling \cite[Theorem 3.5]{EHS08}. From Proposition \ref{10.09.2016--1}-(1), we obtain from $\psi^*$ a morphism 
\begin{equation}\label{22.11.2016--1}
\psi_\#:\pi^{\rm loc}(X')\aro \pi(X\,;\,\sC^{\loc} (X')).
\end{equation} (Recall that $X'$ has a $k$-point $x'$ above $x$.) The translation of  \cite[Theorem 3.5]{EHS08} in our setting  is:
\begin{thm}\label{16.09.2016--5}The  morphism $\psi_\#$ of \eqref{22.11.2016--1} is in fact that kernel of $\pi(X\,;\,\sC^{\loc}(X'))\to G$. Put differently, we have an exact sequence 
\[1\aro \pi^{\loc}(X')\aro\pi(X\,;\,\sC^{\loc}(X'))\aro G\aro 1.\]
\end{thm}
\begin{proof}Firstly, we note that $\psi_\#$ is a closed embedding. So let $E'\in\sC^{\loc}(X')$; by definition $\psi_*(E')$ belongs to $\langle\psi_*\sC^{\loc}(X')\rangle_\ot$ and since the ``co-unit'' $\psi^*(\psi_*E')\to E'$ is an epimorphism,   the criterion \cite[2.21(b), p.139]{dm} immediately proves the statement.  

We then verify that conditions (iii-a) to (iii-c) of Theorem A.1 on p. 396 of \cite{EHS08} are true. In fact, only (iii-a) and (iii-b) need attention, since the argument above already shows that (iii-c) holds.

Let $E\in\sC^{\rm EF}(X)$ become trivial when pulled back to $X'$.  Then, faithfully flat descent shows that $E$ lies in the image of the contracted product $\mathscr L_{X'}$ of \eqref{contracted_product}. Hence, $E$ belongs to $\langle V\rangle_\otimes$. This is condition (iii-a) of \cite[Theorem A1]{EHS08}.

\newcommand{\sA}{\mathcal A}

Let $A$ be the $\sO_X$-coherent algebra $\psi_*(\sO_{X'})$ and let $E$ be an object of $\langle\psi_*\sC^{\loc}(X')\rangle_\otimes$. 
Let $H$ be the space $H^0(X,A\ot_{\sO_X} E)$, $\delta\in H^0(X,A^\ot A^\vee)$ be the global section associated to $\mathrm{id}_A$, and  
\[
\mathrm{ev}:H\ot_kA^\vee\aro E
\]
the evaluation. Since  each $h\in H$ is the image of $h\otimes \delta$ under  \[\mathrm{ev}\ot{\rm id}_A:(H\ot_k A^\vee)\ot_{\sO_X}A\aro E\ot A,\] 
we conclude that 
${\rm ev}\ot{\rm id}_A$ induces a surjection on global sections. This means that $\psi^*({\rm ev})$ induces a surjection on global sections. A fortiori, $\psi^*({\rm Im}({\rm ev}))\to\psi^*E$ induces a surjection on global sections, which implies that  any morphism from $\sO_{X'}$ to $\psi^*E$ factors through $\psi^*({\rm Im}({\rm ev}))$. Now, $\langle\psi_*\sC^{\loc}(X')\rangle_\otimes$ is stable under quotients and $A$ is an object of it; this shows that $\mathrm{Im}({\rm ev})$ lies in $\langle\psi_*\sC^{\loc}(X')\rangle_\otimes$. Then, since $\psi^*(A^\vee)$ is a trivial vector bundle, we can say that $\psi^*({\rm Im}({\rm ev}))$ is equally trivial. In conclusion, $\psi^*(\mathrm{Im}({\rm ev}))$ is the largest trivial subobject of $\psi^*E$, which  is  condition (iii-b) of \cite[Theorem A1]{EHS08}. 

\end{proof}

Now let us order $\sC^{\et}(X)$ in the following way: $W<W'$ if $W\in\langle W'\rangle_\otimes$. Using the direct sum of vector bundles, we see that the resulting partially ordered set is directed, and we obtain a directed system of exact sequences 
\[\xymatrix{&\vdots\ar[d]&\ar[d]\vdots&\ar[d]\vdots  & 
\\
1\ar[r]& \pi^{\loc}(X_{W'})\ar[d]\ar[r]& \pi(X\,;\,\sC^{\loc}(X_{W'}))\ar[d]\ar[r]& \pi(X,W')\ar[r]\ar[d] & 1\\
1\ar[r]& \pi^{\loc}(X_{W})\ar[r]& \pi(X\,;\,\sC^{\loc}(X_W))\ar[r]& \pi(X,W)\ar[r] & 1. }
\]
Taking the limit and using that 
\[
\pi^{\loc}(X_{W'})\aro\pi^{\loc}(X_W)
\]
is always a quotient morphism \cite[Proposition 3.6, p.390]{EHS08}, we arrive at an exact sequence 
\[
1\aro\varprojlim_W \pi^{\loc}(X_{W})\aro \varprojlim_W \pi(X\,;\,\sC^{\loc}(X_W)) \aro \varprojlim_W\pi(X,W)\aro1. 
\]
Note that the rightmost term is a proetale affine group scheme, while the leftmost is a local affine group scheme. In addition, by looking at the categories of representations, we see that the natural morphisms 
\[
\pi(X)\aro \varprojlim_W \pi(X\,;\,\sC^{\loc}(X_W))\quad\text{and}\quad \pi^{\et}(X)\aro \varprojlim_W \pi(X,W) 
\]
are isomorphisms. 
Hence, 
borrowing the notation of \cite[Ch. 6, Exercise 7]{waterhouse}, we conclude that 
\begin{equation}\begin{split}\label{24.11.2016--1}
\pi(X)^o&:=\text{connected component of }\,\pi(X)\\&=\varprojlim\pi^{\loc}(X_W).
\end{split}
\end{equation}
This is precisely  \cite[Theorem 3.5]{EHS08}, as the category $\sD$ appearing on \cite[Definition 3.3]{EHS08} is just the representation category of $\varprojlim_W\pi^{\loc}(X_W)$.

\section{The action of $\pi^\et(X)$ on $\pi(X)^o$}\label{28.11.2016--3}
We work in the setting described in the beginning of section \ref{12.08.2016--1}; in particular,  $k$ is a perfect field of characteristic $p>0$,  $X$ is a proper, reduced and connected algebraic $k$-scheme, and $\psi:X'\to X$ is a torsor under the finite and \'etale group scheme $G$. 

Since the kernel of the morphism 
$\pi(X\,;\,\sC^{\loc}(X'))\to G$
appearing in Theorem \ref{16.09.2016--5} is the local affine group scheme  $\pi^{\loc}(X')$, it is not hard to see, using \cite[6.8, Lemma]{waterhouse}, that 
\[
\pi(X\,;\,\sC^{\loc}(X'))_{\rm red}\stackrel{\sim}{\aro}G.
\]
We then obtain an action of $G$ on $\pi^{\loc}(X')$ by group automorphisms. Our next goal is to understand under which circumstances this action is ``faithful.'' 

\begin{prop}\label{10.09.2016--3}Let $H\subset  G$ be a subgroup scheme acting trivially on $\pi^{\loc}(X')$. Then the natural morphism 
\[
\pi^{\loc}(X')\aro\pi^{\loc}(X'/H)
\]
is an isomorphism. (We use the image of $x'$ on $X'/H$ as base-point for constructing $\pi^{\loc}(X'/H)$.)
\end{prop}

\begin{proof}We  adopt the notations implied by the following diagram:   
\[
\xymatrix{X'\ar@/_25pt/[dd]_\psi\ar[d]^{\rho}\\ X'/H\ar[d]^{\sigma}\\ X.}
\]
Note that  $\rho:X'\to X'/H$ is an $H$-torsor so that we can apply Proposition \ref{10.09.2016--1}-(1) to conclude that $\sigma$ takes objects of $\langle\psi_*\sC^{\loc}(X')\rangle_\ot$ to $\langle\rho_*\sC^{\loc}(X')\rangle_\ot$.  
 
There are now two exact sequence in sight (see Theorem \ref{16.09.2016--5}),
\[\tag{$*$}
1\aro\pi^{\loc}(X')\aro\pi(X'/H\,;\,\sC^{\loc}(X'))\aro H\aro 1
\]
and 
\[\tag{$**$}
1\aro\pi^{\loc}(X')\aro\pi(X\,;\,\sC^{\loc}(X'))\aro G\aro 1.
\]
The above observation assures that they are related  by the commutative diagram
\[
\xymatrix{
1\ar[r] & \pi^{\loc}(X')\ar[r]\ar[d]_\sim&\pi(X'/H\,;\,\sC^{\loc}(X'))\ar[r]\ar[d]_{\sigma_\#} & H\ar[d]_{\text{inclusion}}\ar[r]&1
\\
1\ar[r] & \pi^{\loc}(X')\ar[r]&\pi(X\,;\,\sC^{\loc}(X'))\ar[r] & G\ar[r]&1,
}
\]  
where the arrow $\sigma_\#$ is constructed from the functor $\sigma^*:\langle\psi_*\sC^{\loc}(X')\rangle_\ot\to\langle\rho_*\sC^{\loc}(X')\rangle_\ot$. 
The relevance of this relation is that it shows that the action of $H$ on $\pi^{\loc}(X')$ stemming from the sequence  ($*$)  coincides, once all identifications are unraveled,  with the action of $H$ on $\pi^{\loc}(X')$ derived from ($**$). (The reader wishing to run a careful verification should profit from the fact that the action of $G$, respectively of $H$, is really an action of $\pi(X\,;\,\sC^{\loc}(X'))_{\rm red}$, respectively $\pi(X'/H\,;\,\sC^{\loc}(X'))_{\rm red}$.) The assumption on the statement then implies that the action of $H$ on $\pi^{\loc}(X')$ arising from ($*$) is trivial. From this, we derive a retraction 
\[
r:\pi(X'/H\,;\,\sC^{\loc}(X'))\aro \pi^{\loc}(X')
\]
which exhibits $\pi^{\loc}(X')$ as the largest local quotient of $\pi(X'/H\,;\,\sC^{\loc}(X'))$. But by Proposition \ref{10.09.2016--1}(b), the largest local quotient of $\pi(X'/H\,;\,\sC^{\loc}(X'))$ is   $\pi^{\loc}(X'/H)$, and therefore
\[
\pi^{\loc}(X')\simeq\pi^{\loc}(X'/H).
\] 
(It is not hard to see that this morphism is in fact the canonical one.)
\end{proof}

We now want to show that the conclusion  in the statement of Proposition \ref{10.09.2016--3} \emph{cannot} take place if $X$ is a ``hyperbolic curve''. For that, we only need to study the \emph{largest commutative quotient} of the local fundamental group scheme and apply the following result.  

\begin{prop}\label{10.09.2016--4}Let $C$ be a smooth, geometrically connected and projective one dimensional $k$-scheme (a ``curve''), $c$ a $k$-rational point on $C$, $m$ a positive integer, and $\mathrm{Jac}(C)$ the Jacobian of $C$.   Then, 
the largest quotient of $\pi^{\mathrm{loc}}(C,c)$ which is commutative, finite and annihilated by $p^m$ is isomorphic to $\mathrm{Jac}(C)[p^m]^{\loc}$.
\end{prop}

\begin{proof} To ease notation, we write $J$ in place of $\mathrm{Jac}(C)$. Let 
\[
\varphi:C\aro J
\]
be the Abel-Jacobi (or Albanese) morphism sending $c$ to the origin $e$. Then, we arrive at a commutative diagram 
\[
\xymatrix{\pi(C,c)\ar[r]^{\varphi_\#}\ar[d]^q  & \pi(J,e) \\ \ar[ru]_\alpha \pi(C,c)^{\rm ab},& }
\]
in which the arrow $\alpha$ is an isomorphism \cite[Corollary 3.8]{antei11}. Hence, as explained in Section \ref{16.09.2016--1},  
\[\begin{split}
 \left[ \pi(C,c)^{\rm loc}  \right]^{\rm ab}&\simeq \left[ \pi(C,c)^{\rm ab}  \right]^{\rm loc}\\&\simeq\pi(J,e)^{\loc}.\end{split}
\]
Now let $\bb K$ be the full subcategory of the category of affine group schemes defined by those which are commutative, finite and annihilated by $p^m$. Then, using Lemma \ref{13.09.2016--2} below and the notations of Section \ref{16.09.2016--1}, we see that 
\[
\left[\pi(J,e)^{\loc}\right]^{\bb K}=\left[\pi(J,e)^{\bb K}\right]^{\loc}\simeq  J[p^m]^{\loc}. 
\] 
\end{proof}

\begin{thm}\label{26.09.2016--1}If our $X$ is a smooth, geometrically connected and projective curve of genus at least two,  then no non-trivial subgroup scheme of $G$ acts trivially on $\pi^{\loc}(X')$.  
\end{thm}
\begin{proof}Let $H$ be as in the statement of Proposition \ref{10.09.2016--3}. Then, the fact that $\pi^{\rm loc}(X')$ and $\pi^{\loc}(X'/H)$ are isomorphic implies, via Proposition \ref{10.09.2016--4}, that 
\[
\begin{array}{c}\text{Tangent space}\\\text{at the origin}\end{array}\,\mathrm{Jac}(X')\simeq \begin{array}{c}\text{Tangent space}\\\text{at the origin}\end{array}\,\mathrm{Jac}(X'/H).
\]
Therefore, $X'$ and $X'/H$ have the same genus (which is the dimension of the tangent space to the Jacobian \cite[Proposition 2.1]{milne}). The Riemann-Hurwitz formula then shows that $H$ is trivial.   
\end{proof}

We now wish to obtain from Theorem \ref{26.09.2016--1} a statement which is easier to carry. 

Let $\cg$ be an affine group scheme over $k$ and $M$  a vector space  affording a representation of $\cg$. If $M$ is finite dimensional and $\cg$ is algebraic, we say that $M$ is faithful if the obvious morphism $\cg\to\mathbf{GL}(M)$ is a closed embedding or, equivalently, its kernel is trivial \cite[15.3, Theorem]{waterhouse}.  We now translate this last condition in terms of the coaction $\rho:M\to M\otimes k[\cg]$ for future usage. Define a \emph{modified coefficient} of the representation $M$ as any element of the form
\[  (u\otimes\mathrm{id})  \circ \rho(m)-u(m)\cdot1\in k[\cg],\]
where $m\in M$ and $u\in \mathrm{Hom}(M,k)$. (We leave to the reader the simple task of justifying the term ``modified coefficient''.) Then, the kernel of $\cg\to\mathbf{GL}(M)$ is trivial if and only if the modified coefficients generate the augmentation ideal of $k[\cg]$.

Note that  the definition of modified coefficient makes perfect sense for a general representation, finite or infinite dimensional, of a general affine group scheme $\cg$. Hence, the following encompasses the above definition.

\begin{defn}Let $\cg$ be an affine group scheme and $M$ a vector space affording a representation of $\cg$. We say that $M$ is faithful if the modified coefficient of $M$ generate the augmentation ideal of $k[\cg]$. 
\end{defn}
\begin{rmk}The concept  ``faithful representation'' is not really well established in the literature on group schemes. On the other hand,  a representation $M$ of $\cg$ is faithful if and only if no closed non-trivial subgroup scheme of $\cg$ acts trivially.  (This is because, quite generally, the ideal generated by the modified coefficients is a Hopf-ideal.)
\end{rmk}

Let $A$ be a directed set and $\{\cg_\alpha;q_{\beta\alpha}\}$ 
a projective system indexed by $A$. We assume that the transition morphisms, $q_{\beta\alpha}$ are all faithfully flat and let $\cg$ stand for the limit $\varprojlim\cg_\alpha$. 

\begin{lem}\label{lem.faithful}
Let $M$ be a vector space affording a representation of the affine group scheme $\cg$. 
Assume that for each $\alpha\in A$, there exists some $\beta\ge\alpha$, and a faithful representation $M_\beta$ of $\mathcal G_\beta$ which can be $\cg$-equivariantly embedded in $M$.  Then $M$ is a faithful representation of $\cg$.\qed
\end{lem}
\begin{proof}Let $f$ be an element of the augmentation ideal of $k[\cg]$. Clearly $f$ belongs to the augmentation ideal of some $k[\cg_\alpha]$. Let $M_\beta$ be as in the statement. It then follows that $f$, which also belongs to the augmentation ideal of $k[\cg_\beta]$, can be expressed as a sum $\sum x_i f_i$, where $f_i$ is a modified coefficient of $M_\beta$. Since $M_\beta$ embeds $\cg$-equivariantly in $M$, it is easy to see that each $f_i$ is also a modified coefficient of $M$.   \end{proof}

Employing this language and the identification \eqref{24.11.2016--1}, we can translate Theorem \ref{26.09.2016--1} as follows.

\begin{thm}\label{28.11.2016--2}Let our $X$ be a smooth, geometrically connected and projective curve of genus at least two. Then, the representation of $\pi^\et(X)$ on $k[\pi(X)^o]$ is faithful. \qed
\end{thm}

\section{An exercise on the fundamental group scheme of an abelian variety}\label{12.09.2016--1}
Let   $A$ be an abelian variety over $k$. 
If $m$ and $q$ are positive integers, then multiplication by $q$ on $A[n]$ induces a morphism  $A[qm]\to A[m]$ which is in fact faithfully flat. This allows us to define the affine group scheme $TA:=\varprojlim A[n]$. Paralleling the Lang-Serre theorem \cite[Expos\'e XI,  2.1]{SGA1},  Nori showed in 
 \cite{nori83} that the essentially finite fundamental group scheme 
$\pi(A)$ of $A$ based at the identity is just $TA$. The following is a very simple consequence of this fact. 

\begin{lem}\label{13.09.2016--2}Let $m$ be a positive integer. The obvious arrow $
\pi(A)=TA\to  A[p^m]$   
is universal from $\pi(A)$ to the category of finite, commutative group schemes which are annihilated by $p^m$.
\end{lem}

\begin{proof}To ease notation, let $\pi$ stand for $\pi(A)=TA$. Let $\alpha:\pi\to H$ be an arrow to some $H$ which is finite, commutative and annihilated by $p^m$.  We then have a commutative diagram 
\[
\xymatrix{\pi\ar[r]\ar[rd]_\alpha & A[p^\nu]\ar[d]^\beta\\ &H,}
\]
where $\nu>m$ and the horizontal arrow is the obvious one.
 As $p^m$ annihilates $H$ and $[p^m]:A[p^\nu]\to A[p^{\nu-m}]$ is a quotient   morphism (see \cite[\S8]{milne_AV} for details), we conclude that $A[p^{\nu-m}]\subset\mathrm{Ker}\,\beta$. Using the exact sequence 
\[
0\aro A[p^{\nu-m}]\aro A[p^\nu]\aro A[p^m]\aro0,
\]
we obtain an arrow 
\[
\xymatrix{A[p^\nu]\ar[r]\ar[d] & A[p^m]\ar@{-->}[dl]\\ H.}
\]
This arrow is unique since $\pi\to A[p^m]$ is a quotient morphism. 
\end{proof}

\end{document}